\documentclass[a4paper,12pt,reqno]{amsart}
\usepackage{amsfonts}
\usepackage{amsmath}
\usepackage{amssymb}
\usepackage[a4paper]{geometry}
\usepackage{mathrsfs}

\usepackage[colorlinks]{hyperref}
\renewcommand\eqref[1]{(\ref{#1})} 
\numberwithin{equation}{section}
\theoremstyle{plain}
\newtheorem{thm}{Theorem}[section]
\newtheorem{prop}[thm]{Proposition}

\theoremstyle{definition}

\newtheorem{rem}[thm]{Remark}

\newcommand{\Rn}{\mathbb R^{n}}
\newcommand{\G}{\mathbb G}

\def\e[#1]{{\textrm{e}}^{#1}}

\def\G{{\mathbb G}}

\def\L{\mathcal{L}}

\begin{document}

   \title[Rellich inequalities for sub-Laplacians with drift]
   {Rellich inequalities for sub-Laplacians with drift}

\author[M. Ruzhansky]{Michael Ruzhansky}
\address{
  Michael Ruzhansky:
  \endgraf
  Department of Mathematics
  \endgraf
  Imperial College London
  \endgraf
  180 Queen's Gate, London SW7 2AZ
  \endgraf
  United Kingdom
  \endgraf
  {\it E-mail address} {\rm m.ruzhansky@imperial.ac.uk}
  }

\author[N. Yessirkegenov]{Nurgissa Yessirkegenov}
\address{
  Nurgissa Yessirkegenov:
  \endgraf
  Institute of Mathematics and Mathematical Modelling
  \endgraf
  125 Pushkin str.
  \endgraf
  050010 Almaty
  \endgraf
  Kazakhstan
  \endgraf
  and
  \endgraf
  Department of Mathematics
  \endgraf
  Imperial College London
  \endgraf
  180 Queen's Gate, London SW7 2AZ
  \endgraf
  United Kingdom
  \endgraf
  {\it E-mail address} {\rm n.yessirkegenov15@imperial.ac.uk}
  }

\thanks{The authors were supported in parts by the EPSRC Grants EP/K039407/1 and
EP/R003025/1, and by
 the Leverhulme Research Grant RPG-2017-151, as well as by the MESRK grant 0825/GF4. No new data was collected or
generated during the course of research.}

     \keywords{sub-Laplacian, sub-Laplacians with drift, Rellich inequalities, stratified group, polarizable Carnot group}
     \subjclass[2010]{22E30, 43A80}

     \begin{abstract} In this note we prove horizontal weighted Rellich inequalities for the sub-Laplacian and for sub-Laplacians with drift on general stratified groups. 
     We show how the presence of a drift improves the known inequalities.
     Moreover, we obtain several versions of weighted Rellich inequalities for the sub-Laplacian with drift on the polarizable Carnot groups, also with the weights associated with the fundamental solution of the sub-Laplacian. The obtained results are already new for the Laplacian in the usual Euclidean setting of ${\mathbb R}^n$, embedding the classical Rellich inequality into a family of Rellich inequalities with parameter dependent drifts.
     \end{abstract}
     \maketitle

\section{Introduction}
\label{SEC:intro}
Recall the classical Rellich inequality \cite{R56} for $f\in C_{0}^{\infty}(\Rn\backslash\{0\})$:
\begin{equation}\label{intro_Rel}
\int_{\Rn}|\triangle f(x)|^{2}dx\geq\frac{n^{2}(n-4)^{2}}{16}\int_{\Rn}\frac{|f(x)|^{2}}{|x|^{4}}dx, \quad n\neq2,
\end{equation}
where the constant $\frac{n^{2}(n-4)^{2}}{16}$ is sharp. Note that in the case $n=2$ (see \cite{BEL15}) this inequality holds for functions $f\in C_{0}^{\infty}(\mathbb{R}^{2}\backslash\{0\})$ with the condition
$$\int_{0}^{2\pi}f(r,\theta)\cos\theta d\theta=\int_{0}^{2\pi}f(r,\theta)\sin\theta d\theta=0.$$

In this paper we obtain an embedding of the inequality \eqref{intro_Rel} into a family of inequalities associated to (sub-)Laplacians with drifts.
Such operators have recently attracted much attention in the literature in the context of the spectral multiplier estimates (see \cite{HMM04,HMM05} or recent work \cite{MOV17}).

In particular, in the Euclidean setting of $\Rn$ with $n\geq5$, for all
$a_{i}\in\mathbb{R}$ for $i=1,\ldots,n$ and $\gamma\in \mathbb{R}$, and all
$f\in C_{0}^{\infty}(\Rn\backslash\{0\})$ we prove a new family of inequalities
\begin{multline}\label{intro:Rellich_strat_dr1}
\left\|\left(\triangle+\gamma\sum_{i=1}^{n}a_{i}\partial_{x_{i}}\right) f\right\|^{2}_{L^{2}(\Rn,\mu_{X})}
\geq
\frac{n^2(n-4)^2}{16}\left\|\frac{f}{|x|^{2}}\right\|^{2}_{L^{2}(\Rn,\mu_{X})} \\
+ \gamma^{4}b_{X}^{4}\|f\|^{2}_{L^{2}(\Rn,\mu_{X})}
+\gamma^{2}b_{X}^{2}\frac{(n-2)^{2}}{2}\left\|\frac{f}{|x|}\right\|^{2}_{L^{2}(\Rn,\mu_{X})},
\end{multline}
with the measure $\mu_{X}$ on $\Rn$ given by $d\mu_{X}=e^{-\gamma \sum_{i=1}^{n}a_{i}x_{i}}dx$, where $dx$ is the Lebesgue measure, and $b_{X}=\frac{1}{2}\left(\sum_{j=1}^{n} a_{j}^{2}\right)^{1/2}$. All the constants in \eqref{intro:Rellich_strat_dr1} are sharp
in the sense that there is a sequence of functions such that the equality in \eqref{intro:Rellich_strat_dr1} is attained in the limit of this sequence of functions.

To summarise this and the other results of this paper:
\begin{itemize}
\item the inequality \eqref{intro:Rellich_strat_dr1} reduces to the classical Rellich inequality \eqref{intro_Rel} for $\gamma=0$;
\item for $\gamma\not=0$, the drift term $\gamma\sum_{i=1}^{n}a_{i}\partial_{x_{i}}$, although being of low order, brings an improvement to the Rellich inequality in the form of the addition of two positive terms in \eqref{intro:Rellich_strat_dr1} compared to \eqref{intro_Rel}. Given the drift, the choice of the measure $\mu_{X}$ is essentially unique to make the operator self-adjoint;
\item in addition to \eqref{intro:Rellich_strat_dr1}, we can obtain the weighted family of such inequalities, see \eqref{intro:Rellich_pol};
\item moreover, we obtain a family of horizontal Rellich inequalities for sub-Lap\-la\-cians  with drift on general stratified Lie groups (see Theorem \ref{Rellich_thm_dr}). The starting point for this analysis will be a new weighted Rellich inequality without drift with horizontal weights that we obtain in Theorem \ref{Rellich_thm};
\item in the case of polarizable Carnot groups we obtain a family of Rellich inequalities for the sub-Laplacians with drift for weights also expressed in terms of the fundamental solution of the sub-Laplacian (Theorem \ref{Rellich_thm_pol}).
\end{itemize}

To put the weighted Rellich inequalities in perspective we can recall the following $L^{p}$-Rellich inequality that was obtained by Davies and Hinz \cite{DH98}: Let $V$ be a positive function with the conditions $\triangle V<0$ and $\triangle(V^{\alpha})\leq0$ and, let $\Omega$ be a proper, non-empty open subset of $\Rn$. Then for $a>1$, $p\in (1,\infty)$, and for all functions $f\in C_{0}^{\infty}(\Omega)$, we have
\begin{equation}\label{intro_DH_ineq1}
\int_{\Omega}|\triangle V(x)||f(x)|^{p}dx\leq\left(\frac{p^{2}}{(p-1)a+1}\right)^{p}\int_{\Omega}\frac{(V(x))^{p}}{|\triangle V(x)|^{p-1}}|\triangle f(x)|^{p}dx.
\end{equation}
In the case $\Omega=\Rn\backslash\{0\}$, $a=\frac{n-2}{2\alpha+2}$, $-1<\alpha<\frac{n-4}{2}$, and $V(x)=|x|^{-2\alpha-2}$, the following weighted inequality holds for any function $f\in C_{0}^{\infty}(\Rn\backslash\{0\})$:
\begin{equation}\label{intro_DH_ineq2}
\int_{\Rn}\frac{|f(x)|^{p}}{|x|^{2\alpha+4}}dx\leq\left(\frac{p^{2}}{(n-2\alpha-4)((p-1)n+2\alpha+4-2p)}\right)^{p}\int_{\Omega}\frac{|\triangle f(x)|^{p}}{|x|^{2\alpha+4-2p}}dx.
\end{equation}
We refer to \cite[Corollary 6.2.7]{BEL15} and \cite{EE16} for a Rellich-type inequality where the weight is a mean distance function, and to \cite{Y08} for inequalities on the Heisenberg group, \cite{YH10} on H-type groups, \cite{RS17c} for Rellich inequalities for H\"ormander's sums of squares of vector fields, \cite{SS17} and \cite{RS17a} for Hardy-Rellich inequalities on stratified groups, \cite{K10}, \cite{K06}, \cite{JS11} and \cite{L13} on general Carnot groups with weights associated with the fundamental solution of the sub-Laplacian, and to \cite{RSY16} and \cite{RS16} on general homogeneous groups. See also \cite{AGS06}, \cite{B07}, \cite{BT06} and \cite{TZ07} as well as references therein.

In this paper, we show the weighted Rellich inequalities for the sub-Laplacian with drift. In the case of the Laplacian on $\Rn$ this yields the following new family of inequalities: Let $n\geq3$, $a_{i}\in\mathbb{R}$ for $i=1,\ldots,n$ and let $\alpha,\gamma\in \mathbb{R}$ with $n+2\alpha-4>0$. Then for all functions $f\in C_{0}^{\infty}(\Rn\backslash\{0\})$ we have
$$\left\||x|^{\alpha}\left(\triangle+\gamma\sum_{i=1}^{n}a_{i}\partial_{x_{i}}\right) f\right\|^{2}_{L^{2}(\Rn,\mu_{X})}\geq
\frac{(n+2\alpha-4)^{2}(n-2\alpha)^{2}}{16}\left\||x|^{\alpha-2}f\right\|^{2}_{L^{2}(\Rn,\mu_{X})}$$
\begin{equation}\label{intro:Rellich_pol}
+\gamma^{2}b_{X}^{2}\frac{(n+2\alpha-2)(n-2\alpha-2)}{2}\left\||x|^{\alpha-1}f\right\|^{2}_{L^{2}(\Rn,\mu_{X})}
+\gamma^{4}b_{X}^{4}\left\||x|^{\alpha}f\right\|^{2}_{L^{2}(\Rn,\mu_{X})},
\end{equation}
with the same notations  $\mu_X$ and $b_X$ as in \eqref{intro:Rellich_strat_dr1}.
If $(n-2\alpha)(n-2\alpha-2)\neq0$ with $\alpha\geq1$ and $n\geq2\alpha$, then the constants in \eqref{intro:Rellich_pol} are sharp, in the sense that there is a sequence of functions such that the equality in \eqref{intro:Rellich_pol} is attained in the limit of this sequence of functions.

The plan of the paper as follows. In Section \ref{SEC:prelim} we briefly recall the main concepts of stratified groups, and fix the notation.
In Section \ref{SEC:straf} we show \eqref{intro:Rellich_pol} and \eqref{intro:Rellich_strat_dr1} as a consequence of general results for sub-Laplacians with drift on general stratified groups.
Finally, in Section \ref{SEC:pol} weighted Rellich inequalities \eqref{intro:Rellich_pol} on the polarizable Carnot group are obtained with the weights associated with the fundamental solution of the sub-Laplacian.

\section{Preliminaries}
\label{SEC:prelim}
In this section we briefly recall some notions concerning the setting of stratified groups and discuss some properties of sub-Laplacians with drift.

\subsection{Stratified groups}
We say that a Lie group $\mathbb{G}=(\mathbb{R}^{n}, \circ)$ is a stratified group (or a homogeneous Carnot group) if it satisfies the following:

\begin{itemize}
\item Let $N=N_{1}, N_{2}, \ldots, N_{r}$ be natural numbers such that $N+N_{2}+\ldots+N_{r}=n$. Then  $\mathbb{R}^{n}=\mathbb{R}^{N}\times\ldots\times\mathbb{R}^{N_{r}}$, and for each positive $\lambda$ the dilation $\delta_{\lambda}:\mathbb{R}^{n}\rightarrow\mathbb{R}^{n}$ given by
    $$\delta_{\lambda}(x)=\delta_{\lambda}(x',x^{(2)},\ldots,x^{(r)}):=
    (\lambda x', \lambda^{2}x^{(2)},\ldots,\lambda^{r}x^{(r)})$$ is an automorphism of $\mathbb{G}$, where $x'\equiv x^{(1)}\in\mathbb{R}^{N}$ and $x^{(k)}\in\mathbb{R}^{N_{k}}$ for $k=2,\ldots,r$.
\item For such $N$ as above and for left invariant vector fields $X_{1}, \ldots, X_{N}$ on $\mathbb{G}$ such that $X_{k}(0)=\frac{\partial}{\partial x_{k}}|_{0}$ for $k=1, \ldots, N$, the condition
    $${\rm rank}({\rm Lie}\{X_{1}, \ldots, X_{N}\})=n,$$
is valid, that is, the iterated commutators of $X_{1}, \ldots, X_{N}$ span the Lie algebra of the group $\mathbb{G}$.
\end{itemize}
That is, we say that the triple $\mathbb{G}=(\mathbb{R}^{n}, \circ, \delta_{\lambda})$ is a stratified group, and such groups have been thoroughly investigated by Folland \cite{F75}. We also refer to e.g. \cite{FR16} for more detailed discussions from points of view of groups and span of their Lie algebras.

Here the left invariant vector fields $X_{1}, \ldots, X_{N}$ and $r$ are called (Jacobian) generators of $\mathbb{G}$ and the step of $\mathbb{G}$, respectively. The homogeneous dimension of $\mathbb{G}$ is defined by
$$Q:=\sum_{k=1}^{r}kN_{k}, \;\; N_{1}=N.$$
Let us also recall the fact that the standard Lebesgue measure $dx$ on $\mathbb{R}^{n}$ is the Haar measure for $\mathbb{G}$ (see, e.g. \cite[Proposition 1.6.6]{FR16}). The (canonical) sub-Laplacian on the stratified group $\G$ is defined by
\begin{equation}\label{L}
\L=\sum_{k=1}^{N}X_{k}^{2}.
\end{equation}
The left invariant vector fields $X_{j}$ have an explicit form and satisfy the divergence theorem, see e.g. \cite{RS17b} for the derivation of the exact formula: more precisely, we have
\begin{equation}\label{Xk}
X_{k}=\frac{\partial}{\partial x'_{k}}+\sum_{\ell=2}^{r}\sum_{m=1}^{N_{1}}a_{k,m}^{(\ell)}
(x', \ldots, x^{\ell-1})\frac{\partial}{\partial x_{m}^{(\ell)}},
\end{equation}
see also \cite[Section 3.1.5]{FR16} for a general presentation.

Since we have these explicit representations of the left invariant vector fields $X_{j}$ \eqref{Xk}, one can obtain the identities for any $b \in \mathbb{R}$ and $|x'|\neq0$,
\begin{equation}\label{formula1}
|\nabla_{H}|x'|^{b}|=b|x'|^{b-1},
\end{equation}
and
\begin{equation}\label{formula2}
{\rm div}_{H}\left(\frac{x'}{|x'|^{b}}\right)=\frac{\sum_{j=1}^{N}|x'|^{b}X_{j}x'_{j}-
\sum_{j=1}^{N}x'_{j}b|x'|^{b-1}X_{j}|x'|}{|x'|^{2b}}=\frac{N-b}{|x'|^{b}},
\end{equation}
where $\nabla_{H}:=(X_{1}, \ldots, X_{N})$ is the horizontal gradient, ${\rm div}_{H}\upsilon:=\nabla_{H}\cdot \upsilon$ is the horizontal divergence, and $|x'|=\sqrt{x_{1}^{'2}+\ldots+x_{N}^{'2}}$ is the Euclidean norm on $\mathbb{R}^{N}$.

\subsection{Polarizable Carnot groups}
Let us denote by $u$ the fundamental solution of the sub-Laplacian $\L$, that is
$$-\L u=\delta,$$
where $\delta$ is the Dirac distribution with singularity at the identity element $0$ of $\G$. Then, we know from \cite{F75} that there exists a homogeneous quasi-norm $d$ such that
$$u=d^{2-Q}.$$
So we set
\begin{equation}\label{d}
d(x):=\begin{cases}
(u(x))^{\frac{1}{2-Q}}, \;\;{\rm if}\;\; x\neq0;\\
0, \;\; {\rm if}\;\; x=0.
\end{cases}
\end{equation}
We say that a Carnot group $\G$ (or a stratified Lie group) is polarizable if the homogeneous norm $d=u^{1/(2-Q)}$ satisfies the $\infty$-sub-Laplace equation (see e.g. \cite{BT02})
$$\L_{\G,\infty}d:=\frac{1}{2}\langle\nabla_{H}(|\nabla_{H}d|^{2}),\nabla_{H}d\rangle=0\;\;{\rm in} \;\;\G\backslash\{0\},$$
which implies that (see e.g. \cite[Proposition 2.20]{BT02})
$$\L_{\G,\infty}u=\frac{Q-1}{Q-2}\frac{|\nabla_{H}u|^{4}}{u},$$
or we can write this as
\begin{equation}\label{identity_polar}
\sum_{j=1}^{N}\frac{uX_{j}uX_{j}|\nabla_{H}u|}{|\nabla_{H}u|^{3}}=\frac{Q-1}{Q-2}.
\end{equation}
The Euclidean space, the Heisenberg group and the H-type groups are examples of polarizable Carnot groups.

\subsection{Sub-Laplacians with drift}
Now we briefly recall some necessary facts about sub-Laplacians with drift. Consider on $C_{0}^{\infty}(\mathbb{G})$ the operator
\begin{equation}\label{sub-lap_drift}
\L_{X}:=-\sum_{i,j=1}^{N}a_{i,j}X_{i}X_{j}-\gamma X,
\end{equation}
where $\gamma\in\mathbb{R}$, the matrix $(a_{i,j})_{i,j=1}^{N}$ is real, symmetric, positive definite, and $X\in\mathfrak{g}$ is a left-invariant vector field on $\G$. We can reduce the representation \eqref{sub-lap_drift} to the following form (see e.g. \cite{HMM05} or recent work \cite{MOV17})
\begin{equation}\label{sub-lap_drift1}
\L_{X}=-\sum_{j=1}^{N}X^{2}_{j}-\gamma X:=\L_{0}-\gamma X,
\end{equation}
where $\L_{0}$ is the positive sub-Laplacian on $\G$ defined by
\begin{equation}\label{L0}
\L_{0}=-\sum_{j=1}^{N}X_{j}^{2}.
\end{equation}
If $X=\sum_{j=1}^{N}a_{j}X_{j}$, then we denote $\|X\|:=\left(\sum_{j=1}^{N} a_{j}^{2}\right)^{1/2}$ and
\begin{equation}\label{bX}
b_{X}:=\frac{\|X\|}{2},
\end{equation}
where $a_{i}\in\mathbb{R}$ for $i=1,\ldots,n$.

We will need an analogous result to \cite[Proposition 3.1]{HMM05} for the operator $\L_{X}$, where the case $\gamma=1$ was considered:
\begin{prop}\label{prop_drifted}
Let $\G$ be a connected Lie group, $X_{1},\ldots,X_{N}$ an algebraic basis of $\mathfrak{g}$ and let $X\in\mathfrak{g}\backslash\{0\}$. Let $\gamma\in\mathbb{R}$. Then we have for the operator $\L_{X}$, with domain $C_{0}^{\infty}(\G)$, the following properties:
\begin{itemize}
\item the operator $\L_{X}$ is symmetric on $L^{2}(\G,\mu)$ for some positive measure $\mu$ on $\G$ if and only if there exists a positive character $\chi$ of $\G$ and a constant $C$ such that $\mu=C \mu_{X}$ and $\gamma\nabla_{H}\chi |_{e}=X|_{e}$, where $\mu_{X}$ is the measure absolutely continuous with respect to the Haar measure $\mu$ with density $\chi$;
\item assume that $\gamma\nabla_{H}\chi |_{e}=X|_{e}$ for some positive character $\chi$ of $\G$. Then the operator $\L_{X}$ is essentially self-adjoint on $L^{2}(\G,\mu_{X})$ and its spectrum is contained in the interval $[\gamma^{2}b_{X}^{2},\infty)$.
\end{itemize}
\end{prop}
\begin{proof}[Proof of Proposition \ref{prop_drifted}] Let $\mu$ be a positive measure on $\G$. Then for all test functions $\phi,\psi\in C_{0}^{\infty}(\G)$, we calculate
$$\int_{\G}(\L_{X}\phi)\psi d\mu=-\sum_{j=1}^{N}\left(\int_{\G}(X_{j}^{2}\phi)\psi d\mu\right)-\gamma \int_{\G}\psi X\phi d\mu$$
$$=\sum_{j=1}^{N}\left(\int_{\G}X_{j}\phi X_{j}\psi d\mu+\int_{\G}\psi X_{j}\phi X_{j}\mu \right)
+\gamma\int_{\G}\phi X\psi d\mu+\gamma \int_{\G}\phi \psi X\mu $$
$$=\sum_{j=1}^{N}\left(-\int_{\G}\phi X_{j}^{2}\psi d\mu-2\int_{\G}\phi X_{j}\psi X_{j}\mu -
\int_{\G}\phi \psi X_{j}^{2}\mu \right)$$
$$+\gamma\int_{\G}\phi X\psi d\mu+\gamma \int_{\G}\phi \psi X\mu $$
$$=\int_{\G}\phi(\L_{X}\psi)d\mu+2\gamma\int_{\G}\phi X\psi d\mu
-2\int_{\G}\phi \nabla_{H}\psi \nabla_{H}\mu +\int_{\G}\phi\psi
(\L_{0}+\gamma X)\mu $$
$$=\int_{\G}\phi(\L_{X}\psi)d\mu+\langle\phi,2\gamma(X \psi)\mu-2\nabla_{H}\psi\nabla_{H}\mu+\psi(\L_{0}+\gamma X)\mu\rangle$$
\begin{equation}\label{prop_eq0}
:=\int_{\G}\phi(\L_{X}\psi)d\mu+\mathfrak{I}(\phi,\psi,\mu),
\end{equation}
where $\L_{0}$ is defined in \eqref{L0}, $\nabla_{H}=(X_{1},\ldots,X_{N})$ and $\langle\cdots,\cdots\rangle$ is the pairing between distributions and test functions on $\G$. Then, we see that $\L_{X}$ is symmetric on $L^{2}(\G,\mu)$ if and only if $\mathfrak{I}(\phi,\psi,\mu)=0$ for all functions $\phi$ and $\psi$, that is,
\begin{equation}\label{prop_eq1}
(\L_{0}+\gamma X)\mu=0, \;\;\gamma(X \psi)\mu-\nabla_{H}\psi\nabla_{H}\mu=0,\;\;\forall\psi\in C_{0}^{\infty}(\G).
\end{equation}
We know that the vector fields $X_{1},\ldots,X_{N}$ satisfy H\"{o}rmander's condition, so that $\L_{0}+\gamma X$ is hypoelliptic, which implies with the condition $(\L_{0}+\gamma X)\mu=0$ that $\mu$ has a smooth density $\omega$ with respect to the Haar measure. Then, as in \cite[Proof of the Proposition 3.1]{HMM05}, we show that
\begin{equation}\label{prop_eq11}
X=\sum_{j=1}^{N}a_{j}X_{j},
\end{equation}
for some coefficients $a_{1},\ldots,a_{N}$. Using the fact that $X_{1},\ldots,X_{N}$ are linearly independent and the second equation of \eqref{prop_eq1}, we obtain that
\begin{equation}\label{prop_eq2}
X_{k}\omega=\gamma a_{k}\omega,
\end{equation}
where $k=1,\ldots,N$. Similarly to \cite[Proof of the Proposition 3.1]{HMM05}, we find the solution of \eqref{prop_eq2}
\begin{equation}\label{prop_eq3}
\omega(x)=\omega(e)\exp\left(\gamma\int_{0}^{1}\sum_{j=1}^{N}a_{k}\vartheta_{k}(t)dt\right),
\end{equation}
which is a positive and uniquely determined by its value at the identity, where $\vartheta_{k}(t)$ is the piecewise $C^{1}$ path. By normalising $\omega$, we get that $\omega(e)=1$, and that it is a character of $\G$. Then, we see that the function $x\mapsto \omega(xy)/\omega(y)$ is a solution of \eqref{prop_eq2} for any $y$ in $\G$. Since the value of this function at the identity is $1$, we have $\omega(xy)=\omega(x)\omega(y)$ for any $x,y\in\G$, and $\omega$ is a character of $\G$. From \eqref{prop_eq11} and \eqref{prop_eq2}, we get $\gamma\nabla_{H}\chi |_{e}=X|_{e}$ with $\chi=\omega$. Thus, we have proved the first part of Proposition \ref{prop_drifted}.

As in the case $\gamma=1$ (see \cite[Proposition 3.1]{HMM05}), by considering the isometry $\mathcal{U}_{2}f=\chi^{-1/2}f$ of $L^{2}(\G, \mu)$ onto $L^{2}(\G, \mu_{X})$, we have
\begin{equation}\label{identity_drift}
\chi^{\frac{1}{2}}\L_{X}(\chi^{-\frac{1}{2}}f)=(\L_{0}+\gamma^{2}b_{X}^{2})f,
\end{equation}
which is essentially self-adjoint operator on $L^{2}(\G, \mu)$, where $b_{X}$ is defined in \eqref{bX}. Since the spectrum of this operator is contained in $[\gamma^{2}b_{X}^{2},\infty)$, we obtain that $\L_{X}$ is essentially self-adjoint on $L^{2}(\G, \mu_{X})$ and its spectrum is contained in $[\gamma^{2}b_{X}^{2},\infty)$.

This completes the proof of Proposition \ref{prop_drifted}.
\end{proof}
We assume that there exists a positive character $\chi$ of $\G$ such that
$$\gamma\nabla_{H}\chi |_{e}=X|_{e}.$$
Then, by Proposition \ref{prop_drifted} it is equivalent to the fact that the operator $\L_{X}$ is formally self-adjoint with respect to a positive measure $\widetilde{\mu}$ on $\G$, where $\mu$ is the right Haar measure of $\G$ and $\widetilde{\mu}$ is a multiple of the measure $\mu_{X}=\chi \mu$. In this case we have that $\L_{X}$ is self-adjoint on $L^{2}(\G,\mu_{X})$. We also have for all $f\in C_{0}^{\infty}(\mathbb{G})$
\begin{equation}\label{muX}
L^{2}(\G,\mu)\ni f\mapsto\chi^{-1/2}f\in L^{2}(\G, \mu_{X})
\end{equation}
is an isometric isomorphism. For a detailed discussion about properties of the sub-Laplacians with drift we refer to \cite{HMM04}, \cite{HMM05} and \cite{MOV17}.

\section{Rellich inequalities for sub-Laplacians and for sub-Laplacians with drift on stratified groups}
\label{SEC:straf}
In this section we first show the horizontal weighted Rellich inequality on stratified groups. Then, using this we prove the weighted Rellich inequalities for sub-Laplacians with drift. Let
$$\L=-\L_{0}=\sum_{k=1}^{N}X_{k}^{2}$$ be the usual sub-Laplacian on $\G$.
Let us start from the following weighted Rellich inequality which is essentially a consequence of other results in \cite{RS17a}.

\begin{thm}\label{Rellich_thm}
Let $\mathbb{G}$ be a stratified group with $N\geq3$ being the dimension of the first stratum. Let $\delta\in\mathbb{R}$ with $-N/2\leq\delta\leq-1$. Then for all functions $f\in C_{0}^{\infty}(\mathbb{G}\backslash\{x'=0\})$ we have
\begin{equation}\label{Rellich_strat}
\left\|\frac{\L f}{|x'|^{\delta}}\right\|_{L^{2}(\G)}\geq\left|\frac{(N-2\delta-4)(N+2\delta)}{4}\right|
\left\|\frac{f}{|x'|^{\delta+2}}\right\|_{L^{2}(\G)}.
\end{equation}
If $N+2\delta\neq0$, then the constant in \eqref{Rellich_strat} is sharp.
\end{thm}
\begin{proof}[Proof of Theorem \ref{Rellich_thm}] In the case $p=2$ and $\gamma=2\beta$, \cite[Theorem 3.1]{RS17a} implies
\begin{equation}\label{Hardy_strat0}
\left\|\frac{\nabla_{H} f}{|x'|^{\beta-1}}\right\|_{L^{2}(\G)}
\geq\left|\frac{N-2\beta}{2}\right|\left\|\frac{f}{|x'|^{\beta}}\right\|_{L^{2}(\G)}.
\end{equation}
Putting $\delta+2$ instead of $\beta$, one gets
\begin{equation}\label{Hardy_strat}
\left\|\frac{\nabla_{H} f}{|x'|^{\delta+1}}\right\|_{L^{2}(\G)}\geq\left|\frac{N-2\delta-4}{2}\right|\left\|\frac{f}{|x'|^{\delta+2}}\right\|_{L^{2}(\G)}.
\end{equation}
It is known that the constant $\left|\frac{N-2\delta-4}{2}\right|$ is sharp when $N-2\delta-4\neq0$.
On the other hand, when $p=2$, $\gamma=2\beta$, \cite[Theorem 5.1]{RS17a} implies
\begin{equation}\label{H-Rellich_strat}
\left\|\frac{\L f}{|x'|^{\beta-1}}\right\|_{L^{2}(\G)}
\geq\left|\frac{N+2\beta-2}{2}\right|\left\|\frac{\nabla_{H}f}{|x'|^{\beta}}\right\|_{L^{2}(\G)}
\end{equation}
for $2-N\leq 2\beta\leq0$ and $N\geq3$.
Putting $\delta+1$ instead of $\beta$, we have
\begin{equation}\label{H-Rellich_strat}
\left\|\frac{\L f}{|x'|^{\delta}}\right\|_{L^{2}(\G)}
\geq\left|\frac{N+2\delta}{2}\right|\left\|\frac{\nabla_{H}f}{|x'|^{\delta+1}}\right\|_{L^{2}(\G)}
\end{equation}
for $2-N\leq 2\delta+2\leq0$ and $N\geq3$.
Combining \eqref{Hardy_strat} and \eqref{H-Rellich_strat}, we obtain \eqref{Rellich_strat}.

Now it remains to show sharpness of the constant in \eqref{Rellich_strat}. In \cite[Proof of Theorem 3.1]{RS17a}, the authors showed sharpness of the constant in \eqref{Hardy_strat} by checking the equality condition in H\"{o}lder's inequality. Namely, the function $|x'|^{C_{1}}$ satisfies this equality condition for any real number $C_{1}\neq0$. Similarly, in \cite[Formula 5.7]{RS17a}, checking again the equality condition in H\"{o}lder's inequality, we see that the same function $|x'|^{C_{1}}$ satisfies the equality condition. Therefore, the constant in \eqref{H-Rellich_strat} is sharp when $N+2\delta\neq0$, so, the constant in \eqref{Rellich_strat} is sharp for $N+2\delta\neq0$. 
\end{proof}

Now we prove the weighted Rellich inequalities for sub-Laplacians with drift.
\begin{thm}\label{Rellich_thm_dr}
Let $\mathbb{G}$ be a stratified group with $N\geq3$ being the dimension of the first stratum. Let $\delta\in\mathbb{R}$ with $-N/2\leq\delta\leq-1$. Then for all functions $f\in C_{0}^{\infty}(\mathbb{G}\backslash\{x'=0\})$ we have
$$\left\|\frac{\L_{X} f}{|x'|^{\delta}}\right\|^{2}_{L^{2}(\G,\mu_{X})}\geq
\left(\frac{(N-2\delta-4)(N+2\delta)}{4}\right)^{2}\left\|\frac{f}{|x'|^{\delta+2}}\right\|^{2}_{L^{2}(\G,\mu_{X})}
$$
\begin{equation}\label{Rellich_strat_dr}
+\gamma^{2}b_{X}^{2}\frac{(N-2\delta-2)(N+2\delta-2)}{2}\left\|\frac{f}{|x'|^{\delta+1}}\right\|^{2}_{L^{2}(\G,\mu_{X})}
+\gamma^{4}b_{X}^{4}\left\|\frac{f}{|x'|^{\delta}}\right\|^{2}_{L^{2}(\G,\mu_{X})},
\end{equation}
where $\L_{X}$ and $b_{X}$ are defined in \eqref{sub-lap_drift1} and \eqref{bX}, respectively. If $(N+2\delta)(N+2\delta-2)\neq0$, then the constants in \eqref{Rellich_strat_dr} are sharp.
Moreover, when $\delta=0$ and $N>4$, we have for any function $f\in C_{0}^{\infty}(\mathbb{G}\backslash\{x'=0\})$
\begin{multline}\label{Rellich_strat_dr1}
\left\|\L_{X} f\right\|^{2}_{L^{2}(\G,\mu_{X})}\geq
\left(\frac{N(N-4)}{4}\right)^{2}\left\|\frac{f}{|x'|^{2}}\right\|^{2}_{L^{2}(\G,\mu_{X})}
\\
+\gamma^{2}b_{X}^{2}\frac{(N-2)^{2}}{2}\left\|\frac{f}{|x'|}\right\|^{2}_{L^{2}(\G,\mu_{X})}
+\gamma^{4}b_{X}^{4}\|f\|^{2}_{L^{2}(\G,\mu_{X})},
\end{multline}
with sharp constants. The constants in \eqref{Rellich_strat_dr} and \eqref{Rellich_strat_dr1} are sharp in the sense that there is a sequence of functions such that the equality in \eqref{Rellich_strat_dr} and \eqref{Rellich_strat_dr1} is attained in the limit of this sequence of functions, respectively.
\end{thm}
\begin{rem} When $(N-2\delta-2)(N+2\delta-2)\geq0$, by dropping positive terms in \eqref{Rellich_strat_dr} we get the following `standard' Rellich type inequality for all functions $f\in C_{0}^{\infty}(\mathbb{G}\backslash\{x'=0\})$
\begin{equation}\label{Rellich_strat_dr_rem}\left\|\frac{\L_{X} f}{|x'|^{\delta}}\right\|^{2}_{L^{2}(\G,\mu_{X})}\geq
\left(\frac{(N-2\delta-4)(N+2\delta)}{4}\right)^{2}\left\|\frac{f}{|x'|^{\delta+2}}\right\|^{2}_{L^{2}(\G,\mu_{X})},
\end{equation}
where $\delta\in\mathbb{R}$ with $-N/2\leq\delta\leq-1$ and $N\geq3$.

Similarly, from \eqref{Rellich_strat_dr1} we obtain for any function  $f\in C_{0}^{\infty}(\mathbb{G}\backslash\{x'=0\})$
\begin{equation}\label{Rellich_strat_dr1_rem}\left\|\L_{X} f\right\|^{2}_{L^{2}(\G,\mu_{X})}\geq
\left(\frac{N(N-4)}{4}\right)^{2}\left\|\frac{f}{|x'|^{2}}\right\|^{2}_{L^{2}(\G,\mu_{X})}.
\end{equation}
\end{rem}
\begin{rem} In the Euclidean case $\G =(\Rn, +)$, we have $N=n$, $\nabla_{H}=\nabla=(\partial_{x_{1}},...,\partial_{x_{n}})$ is the usual (full) gradient, setting $X=\sum_{i=1}^{N}a_{i}\partial_{x_{i}}$ and $\delta=-\alpha$, \eqref{Rellich_strat_dr} implies \eqref{intro:Rellich_pol} for $\alpha\geq1$ and $n\geq2\alpha$ with sharp constants, while \eqref{Rellich_strat_dr1} gives \eqref{intro:Rellich_strat_dr1} again with sharp constants.
\end{rem}
\begin{rem} In the case $\gamma=0$, we obtain Theorem \ref{Rellich_thm} and \cite[Formula 1.5]{SS17} from \eqref{Rellich_strat_dr} and \eqref{Rellich_strat_dr1}, respectively.
\end{rem}
\begin{proof}[Proof of Theorem \ref{Rellich_thm_dr}]
Let $g=g(x)\in C_{0}^{\infty}(\G\backslash\{x'=0\})$ be such that $f=\chi^{-1/2}g$. Taking into account \eqref{muX}, we have
$$
\left\|\frac{\L_{X} f}{|x'|^{\delta}}\right\|_{L^{2}(\G,\mu_{X})}=
\left\|\frac{\chi^{1/2}\L_{X}f}{|x'|^{\delta}}\right\|_{L^{2}(\G,\mu)}
=\left\|\frac{\chi^{1/2}\L_{X}(\chi^{-1/2}g)}{|x'|^{\delta}}\right\|_{L^{2}(\G,\mu)}.$$
By this, \eqref{identity_drift} and integration by parts, we calculate
$$
\left\|\frac{\L_{X} f}{|x'|^{\delta}}\right\|^{2}_{L^{2}(\G,\mu_{X})}=\left\|\frac{(\L_{0}+\gamma^{2}b_{X}^{2})g}{|x'|^{\delta}}\right\|^{2}_{L^{2}(\G,\mu)}$$
$$=\left\|\frac{\L_{0} g}{|x'|^{\delta}}\right\|^{2}_{L^{2}(\G,\mu)}+2\gamma^{2}b_{X}^{2}{\rm Re}\int_{\G}\frac{\L_{0}g(x)\overline{g(x)}}{|x'|^{2\delta}}dx
+\gamma^{4}b_{X}^{4}\left\|\frac{g}{|x'|^{\delta}}\right\|^{2}_{L^{2}(\G,\mu)}$$
$$=\left\|\frac{\L_{0} g}{|x'|^{\delta}}\right\|^{2}_{L^{2}(\G,\mu)}-2\gamma^{2}b_{X}^{2}{\rm Re}\sum_{j=1}^{N}\int_{\G}\frac{X_{j}^{2} g(x)\overline{g(x)}}{|x'|^{2\delta}}dx+\gamma^{4}b_{X}^{4}\left\|\frac{g}{|x'|^{\delta}}\right\|^{2}_{L^{2}(\G,\mu)}$$
\begin{multline}\label{strat1}
=\left\|\frac{\L_{0} g}{|x'|^{\delta}}\right\|^{2}_{L^{2}(\G,\mu)}
+2\gamma^{2}b_{X}^{2}\int_{\G}\frac{|\nabla_{H}g(x)|^{2}}{|x'|^{2\delta}}
\\-4\delta \gamma^{2}b_{X}^{2}{\rm Re}\sum_{j=1}^{N}\int_{\G}\frac{x'_{j}X_{j}
g(x)\overline{g(x)}}{|x'|^{2\delta+2}}dx+\gamma^{4}b_{X}^{4}\left\|\frac{g}{|x'|^{\delta}}\right\|^{2}_{L^{2}(\G,\mu)}.
\end{multline}
Since
$${\rm Re}\sum_{j=1}^{N}\int_{\G}\frac{x'_{j}X_{j}
g(x)\overline{g(x)}}{|x'|^{2\delta+2}}dx=(2\delta+2-N)\int_{\G}\frac{|g(x)|^{2}}{|x'|^{2\delta+2}}dx-
{\rm Re}\sum_{j=1}^{N}\int_{\G}\frac{x'_{j}
g(x)\overline{X_{j}g(x)}}{|x'|^{2\delta+2}}dx,$$
we get
$${\rm Re}\sum_{j=1}^{N}\int_{\G}\frac{x'_{j}X_{j}
g(x)\overline{g(x)}}{|x'|^{2\delta+2}}dx=\frac{2\delta+2-N}{2}\int_{\G}\frac{|g(x)|^{2}}{|x'|^{2\delta+2}}dx.$$
Plugging this into \eqref{strat1}, we obtain
\begin{equation}\label{strat2}
\begin{split}
\left\|\frac{\L_{X} f}{|x'|^{\delta}}\right\|^{2}_{L^{2}(\G,\mu_{X})}=
\left\|\frac{\L_{0} g}{|x'|^{\delta}}\right\|^{2}_{L^{2}(\G,\mu)}
+2\gamma^{2}b_{X}^{2}\left\|\frac{\nabla_{H}g}{|x'|^{\delta}}\right\|^{2}_{L^{2}(\G,\mu)}\\
+2\delta(N-2\delta-2)\gamma^{2}b_{X}^{2}\left\|\frac{g}{|x'|^{\delta+1}}\right\|^{2}_{L^{2}(\G,\mu)}
+\gamma^{4}b_{X}^{4}\left\|\frac{g}{|x'|^{\delta}}\right\|^{2}_{L^{2}(\G,\mu)}.
\end{split}
\end{equation}
Using the Rellich \eqref{Rellich_strat} and Hardy \eqref{Hardy_strat0} inequalities, we get from \eqref{strat2}
$$\left\|\frac{\L_{X} f}{|x'|^{\delta}}\right\|^{2}_{L^{2}(\G,\mu_{X})}\geq
\left(\frac{(N-2\delta-4)(N+2\delta)}{4}\right)^{2}\left\|\frac{g}{|x'|^{\delta+2}}\right\|^{2}_{L^{2}(\G,\mu)}
+\gamma^{4}b_{X}^{4}\left\|\frac{g}{|x'|^{\delta}}\right\|^{2}_{L^{2}(\G,\mu)}$$
\begin{equation*}
+2\gamma^{2}b_{X}^{2}\left(\frac{N-2\delta-2}{2}\right)^{2}\left\|\frac{g}{|x'|^{\delta+1}}\right\|^{2}_{L^{2}(\G,\mu)}
+2\delta(N-2\delta-2)\gamma^{2}b_{X}^{2}\left\|\frac{g}{|x'|^{\delta+1}}\right\|^{2}_{L^{2}(\G,\mu)}.
\end{equation*}
It follows that
$$\left\|\frac{\L_{X} f}{|x'|^{\delta}}\right\|^{2}_{L^{2}(\G,\mu_{X})}\geq
\left(\frac{(N-2\delta-4)(N+2\delta)}{4}\right)^{2}\left\|\frac{f}{|x'|^{\delta+2}}\right\|^{2}_{L^{2}(\G,\mu_{X})}
+\gamma^{4}b_{X}^{4}\left\|\frac{f}{|x'|^{\delta}}\right\|^{2}_{L^{2}(\G,\mu_{X})}$$
\begin{equation}\label{strat4}
+\gamma^{2}b_{X}^{2}\frac{(N-2\delta-2)(N+2\delta-2)}{2}\left\|\frac{f}{|x'|^{\delta+1}}\right\|^{2}_{L^{2}(\G,\mu_{X})}.
\end{equation}
As we have discussed in the proof of the Theorem \ref{Rellich_thm}, since the same function satisfies the equality conditions in H\"{o}lder's inequalities, the constants in \eqref{Rellich_strat_dr} are sharp.

To obtain \eqref{Rellich_strat_dr1}, that is the unweighted case $\delta=0$, we use the inequality \eqref{Hardy_strat0} and the following inequality from \cite[Formula 1.5]{SS17}
\begin{equation}\label{Rellich_SS}
\|\L f\|_{L^{2}(\G)}\geq \frac{N(N-4)}{4}\left\|\frac{f}{|x'|^{2}}\right\|_{L^{2}(\G)}, \;\;N\geq5,
\end{equation}
for $f\in C_{0}^{\infty}(\mathbb{G}\backslash\{x'=0\})$. Since it is also known that the constant $\frac{4}{N(N-4)}$ is sharp in \eqref{Rellich_SS}, using the same argument as for the constants in \eqref{Rellich_strat_dr}, we obtain the sharpness of the constants in \eqref{Rellich_strat_dr1}.
\end{proof}

\section{Rellich inequalities for sub-Laplacians and for sub-Laplacians with drift on polarizable Carnot groups}
\label{SEC:pol}

In this section we prove the weighted Rellich inequality for sub-Laplacians with drift on a polarizable Carnot group expressing the weights in terms of the fundamental solution of the sub-Laplacian.

\begin{thm}\label{Rellich_thm_pol}
Let $\mathbb{G}$ be a polarizable Carnot group of homogeneous dimension $Q\geq3$ and let $\theta\in \mathbb{R}$ with $Q+2\theta-4>0$. Then for all functions $f\in C_{0}^{\infty}(\mathbb{G}\backslash\{0\})$ we have
$$\left\|\frac{d^{\theta}}{|\nabla_{H}d|}\L_{X} f\right\|^{2}_{L^{2}(\G,\mu_{X})}
\geq\frac{(Q+2\theta-4)^{2}(Q-2\theta)^{2}}{16}\left\|d^{\theta-2}|\nabla_{H}d|f\right\|^{2}_{L^{2}(\G,\mu_{X})}$$
$$
+\gamma^{4}b_{X}^{4}\left\|\frac{d^{\theta}}{|\nabla_{H}d|}f\right\|^{2}_{L^{2}(\G,\mu_{X})}
+\gamma^{2}b_{X}^{2}\left(\frac{(Q+2\theta-2)(Q-2\theta-2)}{2}\right)\left\|d^{\theta-1}f\right\|^{2}_{L^{2}(\G,\mu_{X})}
$$
$$+2\gamma^{2}b_{X}^{2}(Q-1)(3Q-4)\left\|d^{\theta-1}f\right\|^{2}_{L^{2}(\G,\mu_{X})}$$
\begin{equation}\label{Rellich_pol}
+2\gamma^{2}b_{X}^{2}\int_{\G}
\frac{d^{2\theta+2Q-2}}{|\nabla_{H}d|^{4}}
\left(d^{1-Q}|\nabla_{H}d|\L(d^{1-Q}|\nabla_{H}d|)-3|\nabla_{H}(d^{1-Q}|\nabla_{H}d|)|^{2}\right)|f(x)|^{2}d\mu_{X}(x),
\end{equation}
where $\L_{X}$ and $b_{X}$ are defined in \eqref{sub-lap_drift1} and \eqref{bX}, respectively.
\end{thm}
\begin{rem} In the abelian case $\G =(\Rn, +)$, we have $N=n$, $\nabla_{H}=\nabla=(\partial_{x_{1}},...,\partial_{x_{n}})$ is the usual (full) gradient, $d=|x|$ is the Euclidean distance, hence $|\nabla_{H}d|=1$, and setting $X=\sum_{i=1}^{N}a_{i}\partial_{x_{i}}$, the last two terms in \eqref{Rellich_pol} cancel each other, so that we obtain \eqref{intro:Rellich_pol}.
\end{rem}
\begin{rem} By setting $\gamma=0$ in \eqref{Rellich_pol} so that $\L_{X}=\L_{0}$, we obtain \cite[Theorem 4.1]{K10} on a polarizable Carnot group.
\end{rem}
\begin{proof}[Proof of Theorem \ref{Rellich_thm_pol}]
Let $g=g(x)\in C_{0}^{\infty}(\G\backslash\{0\})$ be such that $f=\chi^{-1/2}g$. By \eqref{muX} we know that $L^{2}(\G, \mu)\in g\mapsto\chi^{-1/2}g\in L^{2}(\G, \mu_{X})$. Then, we have
$$
\left\|\frac{d^{\theta}}{|\nabla_{H}d|}\L_{X} f\right\|_{L^{2}(\G,\mu_{X})}=\left\|\frac{d^{\theta}}{|\nabla_{H}d|}\chi^{1/2}\L_{X}f\right\|_{L^{2}(\G,\mu)}
=\left\|\frac{d^{\theta}}{|\nabla_{H}d|}\chi^{1/2}\L_{X}(\chi^{-1/2}g)\right\|_{L^{2}(\G,\mu)},$$
where $\mu$ is the Haar (Lebesgue) measure on $\G$.
By \eqref{identity_drift} and integration by parts, we calculate
$$
\left\|\frac{d^{\theta}}{|\nabla_{H}d|}\L_{X} f\right\|^{2}_{L^{2}(\G,\mu_{X})}=\left\|\frac{d^{\theta}}{|\nabla_{H}d|}(\L_{0}+\gamma^{2}b_{X}^{2})g\right\|^{2}_{L^{2}(\G,\mu)}$$
$$=\left\|\frac{d^{\theta}}{|\nabla_{H}d|}\L_{0} g\right\|^{2}_{L^{2}(\G,\mu)}+2\gamma^{2}b_{X}^{2}{\rm Re}\int_{\G}
\frac{d^{2\theta}}{|\nabla_{H}d|^{2}}\L_{0} g(x)\overline{g(x)}dx+\gamma^{4}b_{X}^{4}\left\|\frac{d^{\theta}}{|\nabla_{H}d|}g\right\|^{2}_{L^{2}(\G,\mu)}$$
$$=\left\|\frac{d^{\theta}}{|\nabla_{H}d|}\L_{0} g\right\|^{2}_{L^{2}(\G,\mu)}-2\gamma^{2}b_{X}^{2}{\rm Re}\sum_{j=1}^{N}\int_{\G}
\frac{d^{2\theta}}{|\nabla_{H}d|^{2}}X_{j}^{2} g(x)\overline{g(x)}dx+\gamma^{4}b_{X}^{4}\left\|\frac{d^{\theta}}{|\nabla_{H}d|}g\right\|^{2}_{L^{2}(\G,\mu)}$$
$$
=\left\|\frac{d^{\theta}}{|\nabla_{H}d|}\L_{0} g\right\|^{2}_{L^{2}(\G,\mu)}+2\gamma^{2}b_{X}^{2}\left\|\frac{d^{\theta}}{|\nabla_{H}d|}\nabla_{H}g\right\|^{2}_{L^{2}(\G,\mu)}
+\gamma^{4}b_{X}^{4}\left\|\frac{d^{\theta}}{|\nabla_{H}d|}g\right\|^{2}_{L^{2}(\G,\mu)}$$
$$
+2\gamma^{2}b_{X}^{2}{\rm Re}\sum_{j=1}^{N}\int_{\G}
X_{j}g(x)\overline{g(x)}X_{j}\left(\frac{d^{2\theta}}{|\nabla_{H}d|^{2}}\right)dx
$$
$$
=\left\|\frac{d^{\theta}}{|\nabla_{H}d|}\L_{0} g\right\|^{2}_{L^{2}(\G,\mu)}+2\gamma^{2}b_{X}^{2}\left\|\frac{d^{\theta}}{|\nabla_{H}d|}\nabla_{H}g\right\|^{2}_{L^{2}(\G,\mu)}
+\gamma^{4}b_{X}^{4}\left\|\frac{d^{\theta}}{|\nabla_{H}d|}g\right\|^{2}_{L^{2}(\G,\mu)}$$
$$
+4\gamma^{2}b_{X}^{2}\theta{\rm Re}\sum_{j=1}^{N}\int_{\G}
X_{j}g(x)\overline{g(x)}\frac{d^{2\theta-1}X_{j}d}{|\nabla_{H}d|^{2}}dx
-4\gamma^{2}b_{X}^{2}{\rm Re}\sum_{j=1}^{N}\int_{\G}
X_{j}g(x)\overline{g(x)}\frac{d^{2\theta}X_{j}|\nabla_{H}d|}{|\nabla_{H}d|^{3}}dx$$
\begin{equation}\label{polar1}
:=\left\|\frac{d^{\theta}}{|\nabla_{H}d|}\L_{0} g\right\|^{2}_{L^{2}(\G,\mu)}+2\gamma^{2}b_{X}^{2}\left\|\frac{d^{\theta}}{|\nabla_{H}d|}\nabla_{H}g\right\|^{2}_{L^{2}(\G,\mu)}
+\gamma^{4}b_{X}^{4}\left\|\frac{d^{\theta}}{|\nabla_{H}d|}g\right\|^{2}_{L^{2}(\G,\mu)}
+I_{1}+I_{2}.
\end{equation}
Now we need the following Hardy and Rellich inequalities from \cite{K10}:
\begin{thm}[{\cite[Theorem 3.1]{K10}}]
\label{K_Hardy} Let $\G$ be a polarizable Carnot group with homogeneous dimension $Q\geq3$. Let $1<p<Q$ and $\alpha\in\mathbb{R}$ be such that $\alpha>-Q$. Then we have for $f\in C_{0}^{\infty}(\G\backslash\{0\})$
\begin{equation}\label{K_Hardy_ineq}
\int_{\G}d^{\alpha+p}\frac{|\nabla_{H}d\cdot\nabla_{H}f|^{p}}{|\nabla_{H}d|^{2p}}dx\geq \left(\frac{Q+\alpha}{p}\right)^{p}\int_{\G}d^{\alpha}|f|^{p}dx,
\end{equation}
with sharp constant.
\end{thm}
\begin{thm}[{\cite[Theorem 4.1]{K10}} or {\cite[Corollary 3.1]{YKG17}}]
\label{K_Rellich} Let $\G$ be a Carnot group with homogeneous norm $d=u^{1/(2-Q)}$ and $Q\geq3$. Let $\alpha\in\mathbb{R}$ be such that $Q+\alpha-4>0$. Then we have for $f\in C_{0}^{\infty}(\G\backslash\{0\})$
\begin{equation}\label{K_Rellich_ineq}
\int_{\G}\frac{d^{\alpha}}{|\nabla_{H}d|^{2}}|\L_{0} f|^{2}dx\geq \frac{(Q+\alpha-4)^{2}(Q-\alpha)^{2}}{16}\int_{\G}d^{\alpha}
\frac{|\nabla_{H}d|^{2}}{d^{4}}|f|^{2}dx,
\end{equation}
with sharp constant.
\end{thm}
Then, by Theorem \ref{K_Hardy}, one has for $Q+2\theta-2>0$
\begin{equation}\label{polar1_Hardy}
2\gamma^{2}b_{X}^{2}\left\|\frac{d^{\theta}}{|\nabla_{H}d|}\nabla_{H}g\right\|^{2}_{L^{2}(\G,\mu)}\geq
2\gamma^{2}b_{X}^{2}\left(\frac{Q+2\theta-2}{2}\right)^{2}\left\|d^{\theta-1}g\right\|^{2}_{L^{2}(\G,\mu)}.
\end{equation}
Now using Theorem \ref{K_Rellich}, we get for $Q+2\theta-4>0$
\begin{equation}\label{polar1_Rellich}
\left\|\frac{d^{\theta}}{|\nabla_{H}d|}\L_{0} g\right\|^{2}_{L^{2}(\G,\mu)}\geq
\frac{(Q+2\theta-4)^{2}(Q-2\theta)^{2}}{16}\left\|d^{\theta-2}|\nabla_{H}d|g\right\|^{2}_{L^{2}(\G,\mu)}.
\end{equation}
Putting \eqref{polar1_Hardy} and \eqref{polar1_Rellich} into \eqref{polar1}, one obtains for $Q+2\theta-4>0$
$$\left\|\frac{d^{\theta}}{|\nabla_{H}d|}\L_{X} f\right\|^{2}_{L^{2}(\G,\mu_{X})}
\geq\frac{(Q+2\theta-4)^{2}(Q-2\theta)^{2}}{16}\left\|d^{\theta-2}|\nabla_{H}d|g\right\|^{2}_{L^{2}(\G,\mu)}$$
\begin{equation}\label{polar1_HR}
+2\gamma^{2}b_{X}^{2}\left(\frac{Q+2\theta-2}{2}\right)^{2}\left\|d^{\theta-1}g\right\|^{2}_{L^{2}(\G,\mu)}
+\gamma^{4}b_{X}^{4}\left\|\frac{d^{\theta}}{|\nabla_{H}d|}g\right\|^{2}_{L^{2}(\G,\mu)}
+I_{1}+I_{2}.
\end{equation}
Let us calculate $I_{1}$ from \eqref{polar1}:
\begin{equation*}
\begin{split}
I_{1}&=4\gamma^{2}b_{X}^{2}\theta{\rm Re}\sum_{j=1}^{N}\int_{\G}
X_{j}g(x)\overline{g(x)}\frac{d^{2\theta-1}X_{j}d}{|\nabla_{H}d|^{2}}dx\\&=
-4\gamma^{2}b_{X}^{2}\theta{\rm Re}\sum_{j=1}^{N}\int_{\G}
g(x)\overline{X_{j}g(x)}\frac{d^{2\theta-1}X_{j}d}{|\nabla_{H}d|^{2}}dx\\&\quad
-4\gamma^{2}b_{X}^{2}\theta{\rm Re}\sum_{j=1}^{N}\int_{\G}
|g(x)|^{2}X_{j}\left(\frac{d^{2\theta-1}X_{j}d}{|\nabla_{H}d|^{2}}\right)dx.
\end{split}
\end{equation*}
It follows that
\begin{multline}\label{polar2}I_{1}=4\gamma^{2}b_{X}^{2}\theta{\rm Re}\sum_{j=1}^{N}\int_{\G}
X_{j}g(x)\overline{g(x)}\frac{d^{2\theta-1}X_{j}d}{|\nabla_{H}d|^{2}}dx\\=
-2\gamma^{2}b_{X}^{2}\theta\sum_{j=1}^{N}\int_{\G}
|g(x)|^{2}X_{j}\left(\frac{d^{2\theta-1}X_{j}d}{|\nabla_{H}d|^{2}}\right)dx.
\end{multline}
Putting $d=u^{\frac{1}{2-Q}}$ we calculate
\begin{equation*}
\begin{split}
\sum_{j=1}^{N}X_{j}\left(\frac{d^{2\theta-1}X_{j}d}{|\nabla_{H}d|^{2}}\right)&=
(2-Q)\sum_{j=1}^{N}X_{j}\left(\frac{u^{\frac{2\theta-Q}{2-Q}}}{|\nabla_{H}u|^{2}}X_{j}u\right)\\&
=(2-Q)\sum_{j=1}^{N}\frac{2\theta-Q}{2-Q}u^{\frac{2\theta-2}{2-Q}}\frac{(X_{j}u)^{2}}{|\nabla_{H}u|^{2}}
+(2-Q)\sum_{j=1}^{N}\frac{u^{\frac{2\theta-Q}{2-Q}}X_{j}^{2}u}{|\nabla_{H}u|^{2}}\\ & \quad
-2(2-Q)\sum_{j=1}^{N}\frac{u^{\frac{2\theta-Q}{2-Q}}X_{j}u}{|\nabla_{H}u|^{3}}X_{j}|\nabla_{H}u|,
\end{split}
\end{equation*}
which implies, using \eqref{identity_polar}, that
\begin{equation}\label{polar21}
\begin{split}
I_{1}&=-2\gamma^{2}b_{X}^{2}\theta\int_{\G}\left(\sum_{j=1}^{N}X_{j}
\left(\frac{d^{2\theta-1}X_{j}d}{|\nabla_{H}d|^{2}}\right)\right)|g(x)|^{2}dx
\\&
=-2\gamma^{2}b_{X}^{2}\theta(2\theta+Q-2)\int_{\G}u^{\frac{2\theta-2}{2-Q}}
|g(x)|^{2}dx\\&
=-2\gamma^{2}b_{X}^{2}\theta(2\theta+Q-2)\int_{\G}|g(x)|^{2}d^{2\theta-2}dx.
\end{split}
\end{equation}
Now for $I_{2}$ using the integration by parts we obtain
\begin{equation}\label{polar3}
\begin{split}
I_{2}&=-4\gamma^{2}b_{X}^{2}{\rm Re}\sum_{j=1}^{N}\int_{\G}
X_{j}g(x)\overline{g(x)}\frac{d^{2\theta}X_{j}|\nabla_{H}d|}{|\nabla_{H}d|^{3}}dx\\&
=2\gamma^{2}b_{X}^{2}\sum_{j=1}^{N}\int_{\G}
|g(x)|^{2}X_{j}\left(\frac{d^{2\theta}X_{j}|\nabla_{H}d|}{|\nabla_{H}d|^{3}}\right)dx.
\end{split}
\end{equation}
Using $d=u^{\frac{1}{2-Q}}$ one has
$$\sum_{j=1}^{N}X_{j}\left(\frac{d^{2\theta}X_{j}|\nabla_{H}d|}{|\nabla_{H}d|^{3}}\right)
=(2-Q)^{2}\sum_{j=1}^{N}X_{j}\left(\frac{u^{\frac{2\theta-3Q+3}{2-Q}}}
{|\nabla_{H}u|^{3}}X_{j}\left(u^{\frac{Q-1}{2-Q}}|\nabla_{H}u|\right)\right)$$
$$=(2-Q)^{2}\sum_{j=1}^{N}X_{j}\left(\frac{Q-1}{2-Q}u^{\frac{2\theta-Q}{2-Q}}\frac{X_{j}u}{|\nabla_{H}u|^{2}}
+u^{\frac{2\theta-2Q+2}{2-Q}}\frac{X_{j}|\nabla_{H}u|}{|\nabla_{H}u|^{3}}\right)$$
\begin{equation}\label{polar31}
:=(2-Q)^{2}J_{1}+(2-Q)^{2}J_{2}.
\end{equation}
Then, taking into account \eqref{identity_polar} we have for $J_{1}$
$$J_{1}=\sum_{j=1}^{N}X_{j}\left(\frac{Q-1}{2-Q}u^{\frac{2\theta-Q}{2-Q}}\frac{X_{j}u}{|\nabla_{H}u|^{2}}\right)
=\frac{(Q-1)(2\theta-Q)}{(Q-2)^{2}}u^{\frac{2\theta-2}{2-Q}}\sum_{j=1}^{N}\frac{(X_{j}u)^{2}}{|\nabla_{H}u|^{2}}$$
$$+\frac{Q-1}{2-Q}u^{\frac{2\theta-Q}{2-Q}}\sum_{j=1}^{N}\frac{X_{j}^{2}u}{|\nabla_{H}u|^{2}}
-\frac{2(Q-1)}{(2-Q)}u^{\frac{2\theta-Q}{2-Q}}\sum_{j=1}^{N}\frac{X_{j}uX_{j}|\nabla_{H}u|}{|\nabla_{H}u|^{3}}$$
\begin{equation}\label{polar32}
=\frac{(Q-1)(2\theta-Q)}{(Q-2)^{2}}u^{\frac{2\theta-2}{2-Q}}
+\frac{Q-1}{2-Q}u^{\frac{2\theta-Q}{2-Q}}\frac{\L u}{|\nabla_{H}u|^{2}}
+\frac{2(Q-1)^{2}}{(Q-2)^{2}}u^{\frac{2\theta-2}{2-Q}}.
\end{equation}
Now we calculate for $J_{2}$
$$J_{2}=\sum_{j=1}^{N}X_{j}\left(u^{\frac{2\theta-2Q+2}{2-Q}}\frac{X_{j}|\nabla_{H}u|}{|\nabla_{H}u|^{3}}\right)
=\frac{2\theta-2Q+2}{2-Q}u^{\frac{2\theta-Q}{2-Q}}\sum_{j=1}^{N}\frac{X_{j}uX_{j}|\nabla_{H}u|}{|\nabla_{H}u|^{3}}$$
$$+u^{\frac{2\theta-2Q+2}{2-Q}}\sum_{j=1}^{N}\frac{X_{j}^{2}|\nabla_{H}u|}{|\nabla_{H}u|^{3}}
-3u^{\frac{2\theta-2Q+2}{2-Q}}\sum_{j=1}^{N}\frac{(X_{j}|\nabla_{H}u|)^{2}}{|\nabla_{H}u|^{4}}$$
\begin{equation}\label{polar33}
=\frac{2\theta-2Q+2}{2-Q}\left(\frac{Q-1}{Q-2}\right)u^{\frac{2\theta-2}{2-Q}}
+u^{\frac{2\theta-2Q+2}{2-Q}}\frac{\L|\nabla_{H}u|}{|\nabla_{H}u|^{3}}
-3u^{\frac{2\theta-2Q+2}{2-Q}}\frac{|\nabla_{H}|\nabla_{H}u||^{2}}{|\nabla_{H}u|^{4}},
\end{equation}
where we have used \eqref{identity_polar} in the last equality. Plugging \eqref{polar32} and \eqref{polar33}
into \eqref{polar31}, we obtain
$$\sum_{j=1}^{N}X_{j}\left(\frac{d^{2\theta}X_{j}|\nabla_{H}d|}{|\nabla_{H}d|^{3}}\right)
=(Q-1)(3Q-4)u^{\frac{2\theta-2}{2-Q}}
+(2-Q)(Q-1)\frac{\L u}{|\nabla_{H}u|^{2}}
$$$$
+(Q-2)^{2}u^{\frac{2\theta-2Q+2}{2-Q}}|\nabla_{H}u|^{-4}(|\nabla_{H}u|\L|\nabla_{H}u|-3|\nabla_{H}|\nabla_{H}u||^{2}).$$
Then, we have for $I_{2}$
$$I_{2}
=2\gamma^{2}b_{X}^{2}\sum_{j=1}^{N}\int_{\G}
|g(x)|^{2}X_{j}\left(\frac{d^{2\theta}X_{j}|\nabla_{H}d|}{|\nabla_{H}d|^{3}}\right)dx$$
$$=2\gamma^{2}b_{X}^{2}(Q-1)(3Q-4)\int_{\G}
u^{\frac{2\theta-2}{2-Q}}|g(x)|^{2}dx$$
$$+2\gamma^{2}b_{X}^{2}(Q-2)^{2}\int_{\G}
u^{\frac{2\theta-2Q+2}{2-Q}}|\nabla_{H}u|^{-4}(|\nabla_{H}u|\L|\nabla_{H}u|-3|\nabla_{H}|\nabla_{H}u||^{2})|g(x)|^{2}dx.$$
Setting here $u=d^{2-Q}$, we get for $I_{2}$
\begin{equation}\label{polar9}
\begin{split}I_{2}&=
2\gamma^{2}b_{X}^{2}(Q-1)(3Q-4)\int_{\G}
d^{2\theta-2}|g(x)|^{2}dx
\\&
+2\gamma^{2}b_{X}^{2}\int_{\G}
\frac{d^{2\theta+2Q-2}}{|\nabla_{H}d|^{4}}
\left(d^{1-Q}|\nabla_{H}d|\L(d^{1-Q}|\nabla_{H}d|)-3|\nabla_{H}(d^{1-Q}|\nabla_{H}d|)|^{2}\right)|g(x)|^{2}dx.
\end{split}
\end{equation}
Thus, by \eqref{polar21} and \eqref{polar9}, one has
$$I_{1}+I_{2}=2\gamma^{2}b_{X}^{2}((Q-1)(3Q-4)-\theta(2\theta+Q-2))\left\|d^{\theta-1}g\right\|^{2}_{L^{2}(\G,\mu)}
$$
\begin{equation*}\label{polar10}
+2\gamma^{2}b_{X}^{2}\int_{\G}
\frac{d^{2\theta+2Q-2}}{|\nabla_{H}d|^{4}}
\left(d^{1-Q}|\nabla_{H}d|\L(d^{1-Q}|\nabla_{H}d|)-3|\nabla_{H}(d^{1-Q}|\nabla_{H}d|)|^{2}\right)|g(x)|^{2}dx.
\end{equation*}
Combining this with \eqref{polar1_HR} and taking into account \eqref{muX} we obtain \eqref{Rellich_thm_pol}.
\end{proof}

\end{document}